\documentclass{article}
\usepackage{amsmath,amssymb,amsthm,graphics}
\usepackage[mathscr]{eucal}
\title{Degrees that are not degrees of categoricity}

\author{Bernard A. Anderson\\ \small{Department of Mathematics and Physical Sciences}\\ \small{Gordon State College}\\ \small{banderson@gordonstate.edu}\\ \small{www.gordonstate.edu/Faculty/banderson} \and
Barbara F.~Csima \thanks{B.\ Csima was partially supported by Canadian NSERC
Discovery Grant 312501.} \\ \small{Department of Pure Mathematics}\\ \small{University of
Waterloo}\\
\small{csima@math.uwaterloo.ca}\\
\small{www.math.uwaterloo.ca/$\sim$csima}}

\newtheorem{theorem}{Theorem}[section]
\newtheorem{lemma}[theorem]{Lemma}
\newtheorem{cor}[theorem]{Corollary}
\newtheorem{prop}[theorem]{Proposition}
\newtheorem{remark}[theorem]{Remark}

\theoremstyle{definition}
\newtheorem{definition}{Definition}

\newcommand{\setsep}{\ensuremath{\, | \;}}

\newcommand{\strings}{2^{<\omega}}
\newcommand{\num}{\in\omega}

\newcommand{\baistr}{\omega^{<\omega}}

\newcommand{\empstr}{\langle \rangle}
\newcommand{\rstrd}{\upharpoonright\!\!\upharpoonright}
\newcommand{\Eres}{\upharpoonright\!\!\upharpoonright}
\newcommand{\es}{\emptyset}
\newcommand{\ce}{c.e.\ }
\newcommand{\la}{\langle}
\newcommand{\ra}{\rangle}
\newcommand{\rt}{\rightarrow}

\newcommand{\dg}[1]{\mathbf{#1}}
\newcommand{\sta}{\mathcal{A}}
\newcommand{\stm}{\mathcal{M}}
\newcommand{\stb}{\mathcal{B}}

\newcommand{\len}{\mbox{length}}
\newcommand{\esdj}{\emptyset^{\prime \prime}}
\newcommand{\zdj}{\mathbf{0}^{\prime \prime}}
\newcommand{\scrA}{\mathcal{A}}
\newcommand{\scrB}{\mathcal{B}}

\newcommand{\scrN}{\mathcal{N}}
\newcommand{\iso}{\cong}
\newcommand{\up}{\uparrow}
\newcommand{\conc}{\widehat{\ }}
\newcommand{\nin}{\not\in}
\newcommand{\res}{\upharpoonright}
\newcommand{\N}{\mathbb{N}}

\begin{document}
\maketitle

\begin{abstract} A computable structure $\sta$ is
$\dg{x}$-{\it computably categorical} for some Turing degree
$\dg{x}$, if for every computable structure $\stb \cong \sta$
there is an isomorphism $f:\stb \to \sta$ with $f \leq_T
\dg{x}$.  A degree $\dg{x}$ is a {\it degree of categoricity}
if there is a computable structure $\sta$ such that $\sta$ is
$\dg{x}$-computably categorical, and for all $\dg{y}$, if
$\sta$ is $\dg{y}$-computably categorical then $\dg{x} \leq_T
\dg{y}$.

We construct a $\Sigma^0_2$ set whose degree is not a degree of
categoricity.  We also demonstrate a large class of degrees
that are not degrees of categoricity by showing that every
degree of a set which is 2-generic relative to some perfect
tree is not a degree of categoricity.  Finally, we prove that
every noncomputable hyperimmune-free degree is not a degree of
categoricity.
\end{abstract}

\section{Introduction}

Classically, isomorphic structures are considered to be
equivalent. In computable structure theory, one has to be more
careful. Different copies of the same structure may have
different complexity, and for some structures, it can happen
that there are two computable copies of the structure between
which there is no computable isomorphism. In fact, for
situations where this does not happen, we have the following
definition.

\begin{definition} A computable structure $\scrA$ is
\emph{computably categorical} if for all computable $\scrB \iso
\scrA$ there exists a computable isomorphism between $\scrA$
and $\scrB$.
\end{definition}

For example, any two computable dense linear orders without
endpoints are computably isomorphic. Thus, any computable dense
linear order without endpoints is computably categorical.

On the other hand, it is well known that the structure $(\N,
<)$, the natural numbers with the usual $<$ order, is not
computably categorical. Indeed, let $\{K_s\}_{s\in \omega}$ be
a computable enumeration of $\es'$ where there is exactly one
element enumerated at each stage, and consider the order
$\scrA$ where the even numbers have their usual order and $2n
<_\scrA 2s+1 <_\scrA 2n +2$ iff $n \in K_{s+1} - K_s$. Note
that any isomorphism $f:\scrA \rt \scrN$ computes $\es'$.
Conversely, between any two computable copies of $(\N, <)$
there is a $\es'$-computable isomorphism. That is, it seems
that ${\bf 0'}$ is the degree of difficulty associated to the
problem of computing isomorphisms between arbitrary copies of
$(\N, <)$. This motivates the following definitions.

\begin{definition} A computable structure $\scrA$ is
${\bf d}$-\emph{computably categorical} if for all computable
$\scrB \iso \scrA$ there exists a {\bf d}-computable
isomorphism between $\scrA$ and $\scrB$.
\end{definition}

So with this definition, $(\N, <)$ is ${\bf 0'}$-computably
categorical.

\begin{definition} We say a structure $\scrA$ has \emph{degree of
categoricity} $\bf{d}$ if $\scrA$ is $\bf{d}$-computably
categorical, and for all $\bf{c}$ such that $\scrA$ is
$\bf{c}$-computably categorical, $\bf{d} \leq \bf{c}$. We say a
degree $\bf{d}$ is a \emph{degree of categoricity} if there is
some structure with degree of categoricity $\bf{d}$.
\end{definition}

So in our examples, we have seen that ${\bf 0}$ and ${\bf 0'}$
are degrees of categoricity.

The notion of a degree of categoricity was first introduced by
Fokina, Kalimul\-lin and Miller in \cite{FKM}. In that paper,
they showed that if ${\bf d}$ is d.c.e.\ in and above
${\bf0^{(n)}}$, then ${\bf d}$ is a degree of categoricity.
They also showed that ${\bf 0^{(\omega)}}$ is a degree of
categoricity. In fact, all the examples they constructed had
the following, stronger property.

\begin{definition}A degree of categoricity ${\bf d}$ is a
\emph{strong} degree of categoricity if there is a structure
$\scrA$ with computable copies $\scrA_0$ and $\scrA_1$ such
that ${\bf d}$ is the degree of categoricity for $\scrA$, and
every isomorphism $f: \scrA_0 \rightarrow \scrA_1$ satisfies
deg$(f) \geq {\bf d}$.
\end{definition}

Fokina, Kalimullin and Miller \cite{FKM} showed that all strong
degrees of categoricity are hyperarithmetical. Later, Csima,
Franklin and Shore \cite{CFS} showed that in fact all degrees
of categoricity are hyperarithmetical. This may or may not be
an improvement, as it is unknown whether all degrees of
categoricity are strong.

Csima, Franklin and Shore \cite{CFS} have shown that for every
computable ordinal $\alpha$, ${\bf 0^{(\alpha)}}$ is a (strong)
degree of categoricity. They also showed that if $\alpha$ is a
computable successor ordinal and ${\bf d}$ is d.c.e.\ in and
above ${\bf 0^{(\alpha)}}$, then ${\bf d}$ is a (strong) degree
of categoricity.

The work on degrees of categoricity so far has gone into
showing that various degrees \emph{are} degrees of
categoricity. In this paper we address the question: What are
examples of degrees that are \emph{not} degrees of
categoricity? Certainly, as there are only countably many
computable structures, there are only countably many degrees of
categoricity. In section 3, we give a basic construction of a
degree below ${\bf 0''}$ that is not a degree of categoricity.
In section 4 we show that degrees of 2-generics (indeed, of
2-generics relative to perfect trees) are not degrees of
categoricity. In section 5 we show that noncomputable
hyperimmune-free degrees cannot be degrees of categoricity.
Finally, in section 6, we show that there exists a $\Sigma^0_2$
degree that is not a degree of categoricity.

\section{Notation}

For general references, see Harizanov \cite{h98} for computable
structure theory, and Soare \cite{Soare} for computability
theory.

We use $T$ to denote a tree (a subset of $\strings$ closed
under initial segments).  All other uppercase letters are used
for subsets of the natural numbers and lowercase bold letters
are used for Turing degrees.  We use $\alpha$, $\beta$,
$\sigma$, and $\tau$ to represent strings (elements of
$\strings$).  When dealing with strings, $\supseteq$
denotes string extension. For $A \subseteq \N$ and $n \in
\omega$, we let $A\res n = \{ x \in A \mid x < n \}$ and let $A
\Eres n = \{ x \in A \mid x \leq n \}$, with analogous
definitions of $\sigma \res n$ and $\sigma \Eres n$ for $\sigma
\in 2^{<\omega}$. We have $\Phi_n$ denote the $n$-th
oracle Turing reduction and $\varphi_n$ the $n$-th Turing
reduction.

We use calligraphic letters ($\sta, \stb, \stm$) to denote
computable structures.  We let $\sta_n$ denote the $n$-th
partial computable structure under some effective listing.  For
simplicity, we assume all computable structures have domain
$\omega$ or an initial segment of $\omega$.  We also assume all
structures are in a finite language.

We let Part$(\scrA,\scrB)$ denote the set of partial
isomorphisms between $\scrA$ and $\scrB$. That is, the set of
functions that, on their domain/range, are injective
homomorphisms from a substructure of $\scrA$ to a substructure
of $\scrB$. Note that if $\scrA_i$ and $\scrA_j$ are computable
structures according to our listing, then if for some $\sigma$
and some $t$ we have $\Phi_{e,t}^\sigma \nin\
$Part$(\scrA_{i,t}, \scrA_{j,t})$, then for all $A\supset
\sigma$, $\Phi_e^A \nin\ $Part$(\scrA_i, \scrA_j)$. Also, if
$f$ is a bijection such that for all $n$, $f\res n \in\
$Part$(\scrA, \scrB)$, then $f: \scrA \iso \scrB$.

For $\sigma, \tau \in 2^{<\omega}$ we write $\sigma <_L \tau$
if $\sigma \subset \tau$ or if there is some $n$ such that
$\sigma(n)=0$, $\tau(n)=1$, and for all $k< n$,
$\sigma(k)=\tau(k)$. That is, if $\sigma$ comes before $\tau$
in the usual lexicographical order. We say a set $A$ is
left-c.e.\ if there is a computable sequence $\alpha_s \in
2^{<\omega}$, where $A(n) = \lim_s \alpha_s(n)$, and for all
$s$, $\alpha_s <_L \alpha_{s+1}$. Left c.e.\ sets have been
studied extensively in the area of algorithmic randomness (see
Nies \cite{Nies}). It is easy to see that all c.e.\ sets are
also left-c.e. The converse does not hold. However, every
left-c.e.\ set is Turing equivalent to a c.e.\ set (since $A
\equiv_T \{ \sigma \mid \sigma <_L A\}$), and it is this
feature of left-c.e.\ sets that we will make use of.

Finally, we use the following definition.

\begin{definition} Let $\scrA$ be a computable structure.
We define CatSpec$(\scrA)$ to be the set of degrees ${\bf d}$ such
that $\scrA$ is ${\bf d}$-computably categorical.
\end{definition}

\section{Basic construction}

It will follow from several of the results in this paper that
there is a degree $\dg{x} \leq_T \zdj$ which is not a degree of
categoricity.  However we will briefly sketch a proof of this
fact here, since the ideas we use are expanded on in the proofs
in sections \ref{gensec} and \ref{sigsec}.

For the proof, we will construct a noncomputable set $X$ such
that for all $m,k$ either $X$ does not compute an isomorphism
from $\sta_m$ to $\sta_k$, or there is a computable isomorphism
from $\sta_m$ to $\sta_k$.  Given the construction, suppose the
degree $\dg{x}$ of $X$ is a degree of categoricity, witnessed
by $\sta$.  Let $\stb$ be an arbitrary computable copy of
$\sta$.  Since $\sta$ is $\dg{x}$-computably categorical, $X$
computes an isomorphism from $\stb$ to $\sta$.  By the
construction, there is then a computable isomorphism from
$\stb$ to $\sta$.  Since $\stb$ was arbitrary, $\sta$ is
computably categorical, for a contradiction.

We will build $X$ by finite extensions using a $\esdj$ oracle.
At each stage we will use the $\esdj$ oracle to try to extend
$X$ to block some $\Phi^X_l$ from being an isomorphism.  If
such a block is not possible, we will argue that a computable
isomorphism can be found.

\begin{prop} There is a degree $\dg{x} \leq_T \zdj$ such that $\dg{x}$ is not a degree of categoricity. \label{Baseprop} \end{prop}
\begin{proof}[Proof (sketch)] We build $X$ by finite extensions
using a $\esdj$ oracle. We start with $X_0 = \empstr$.  For
stage $s+1$, we will ensure that either $\Phi^X_l$ is not an
isomorphism from $\sta_m$ to $\sta_k$ or there is a computable
isomorphism from $\sta_m$ to $\sta_k$, where $s = \la l,m,k
\ra$.

We start stage $s+1$ by using $\es^\prime$ to diagonalize
against $X$ being computable by $\varphi_s$.  We then use
$\es^\prime$ to determine if there is a $\sigma \supseteq X_s$
and a time $t$ such that $\Phi^\sigma_{l,t}$ can be seen not to
be an injective homomorphism from $\sta_m$ to $\sta_k$.  If
there is, we let $X_{s+1} = \sigma$ and proceed to the next
stage.  If there is not, we ask $\esdj$ if there exists $\sigma
\supseteq X_s$ and $n \num$ such that for all $\tau \supseteq
\sigma$ we have $\Phi^\tau_l$ omits $n$ from its domain or
range.  If such a $\sigma$ exists, we note $\Phi^Y_l$ is not an
isomorphism for any  $Y \supseteq \sigma$, so we let $X_{s+1} =
\sigma$.

If the answer to both questions is no, then for any $\gamma
\supseteq X_s$ we have that $\Phi^\gamma_l$ is a partial
injective homomorphism and for every $n$ there is a $\tau
\supseteq \gamma$ with $n$ in the domain and range of
$\Phi^\tau_l$.  Note this $\tau$ also extends $X_s$, so these
properties also hold for $\tau$.  We let $\alpha_0 = X_s$ and
$\alpha_{n+1}$ be the first extension of $\alpha_n$ which puts
$n$ into the domain and range of $\Phi_l^{\alpha_{n+1}}$.
Letting $A = \bigcup_{n \in \omega} \alpha_n$ and $f =
\Phi^A_l$, we have that $f$ is a computable isomorphism from
$\sta_m$ to $\sta_k$.  Thus we let $X_{s+1} = X_s$ and move to
the next stage.

This completes our construction of $X$.  We have $X \leq_T
\esdj$, and as noted in the explanation before the proof, the
degree $\dg{x}$ of $X$ is not a degree of categoricity.
\end{proof}

\section{2-generic relative to some perfect tree} \label{gensec}

We wish to generalize Proposition \ref{Baseprop} to show that a
large class of sets have degrees that are not degrees of
categoricity.  To do this we will use the concept of sets that
are $n$-generic relative to some perfect tree.  Recall a set
$G$ is $n$-generic if for every $\Sigma^0_n$ set, either it meets
the set or some initial segment cannot be extended to meet the
set.

\begin{definition} A set $G$ is $n$-generic if for every $\Sigma^0_n$
 subset $S$ of $\strings$, either there is an $l$ such that
 $G \rstrd l \in S$, or there is an $l$ such that for all $\sigma \supseteq G \rstrd l$ we have $\sigma \notin S$. \end{definition}

We relativize this notion from $\strings$ to a perfect tree.

\begin{definition} A set $G$ is $n$-generic relative to the perfect
tree $T$ if $G$ is a path through $T$ and for every $\Sigma^0_n
(T)$ subset $S$ of $\strings$, either there is an $l$ such that
$G \rstrd l \in S$, or there is an $l$ such that for all
$\sigma \supseteq G \rstrd l$ with $\sigma \in T$ we have
$\sigma \notin S$. \end{definition}

\begin{definition} A set $G$ is $n$-generic relative to some
perfect tree if there exists a perfect tree $T$ such that $G$ is $n$-generic relative to $T$. \end{definition}

It has been shown that almost all sets are 2-generic relative to some perfect tree.

\begin{theorem}[Anderson \cite{andgen}] For any $n$, all but countably many sets are $n$-generic relative to some perfect tree. \end{theorem}

We will prove that every degree containing a set that is
2-generic relative to some perfect tree is not a degree of
categoricity.  As a result we are able to limit the degrees of
categoricity to an easily defined countable class (distinct
from HYP).  It will follow as a corollary that for any degree
$\dg{x}$ there is a degree $\dg{y}$ with $\dg{x} \leq_T \dg{y}
\leq_T \dg{x}^{\prime \prime}$ such that $\dg{y}$ is not a
degree of categoricity.

We can view the proof of our theorem as the proof to
Proposition \ref{Baseprop} relativized twice, in successive
stages.  We first relativize the proof from a single
$\Delta^0_3$ set to any 2-generic set.  The idea is that if $G$
computes an isomorphism then we can find an initial segment $G
\rstrd l$ such that for every extension of $G \rstrd l$ the
answer to both of the questions we ask in the original proof is
no.  We can then build a computable isomorphism as in the
original proof.

We next relativize from every 2-generic to every 2-generic
relative to an arbitrary perfect tree $T$.  The key here is
that the proof of Proposition \ref{Baseprop} is stronger than required.  In the
original proof, we show that if $G$ computes an isomorphism
then there is a computable one.  It suffices to fix some $H
\not\geq_T G$ such that if $G$ computes an isomorphism then so
does $H$.  When we relativize to $T$ we obtain this for $T$ in
the place of $H$.

\begin{theorem} Let $G$ be 2-generic relative to some perfect tree.
Then the degree of $G$ is not a degree of categoricity. \label{Grpt} \end{theorem}
\begin{proof} Let $G$ be 2-generic relative to the perfect tree $T$.
Suppose the degree of $G$ is a degree of categoricity,
witnessed by $\sta$.  Let $\stb$ be an arbitrary computable
structure such that $\sta$ is isomorphic to $\stb$.  We will
show there is an isomorphism $f: \sta \to \stb$ with $f \leq_T
T$.  Since our choice of $\stb$ is arbitrary, we can then
conclude $T \in$ CatSpec($\sta$), so $T \geq_T G$ for a
contradiction.

Let $\Psi$ be such that $\Psi^G$ is an isomorphism from $\sta$
to $\stb$.  Let $R_s$ be the set of strings $\sigma$ such that
$\Psi_s^\sigma$ contains values contradicting it being a
partial injective homomorphism from $\sta$ to $\stb$, i.e.\ $\Psi_s^\sigma \notin \text{Part}(\sta, \stb)$.  We note
$R_s$ is (uniformly) computable.  Let
\begin{eqnarray*} & S = \{\sigma \in \strings \setsep \exists n \, \forall  s \, \forall \tau \in T [\tau \supseteq \sigma \to & \\ & (\tau \in R_s \ \vee\  n \notin \text{dom}(\Psi_s^\tau) \ \vee\  n \notin \text{ran}(\Psi_s^\tau))]\} & \end{eqnarray*}

We note $S$ is $\Sigma^0_2 (T)$.   Suppose for some $j$ we have
$G \rstrd j \in S$, witnessed by $n$.  Since $\Psi^G$ is an
isomorphism from $\sta$ to $\stb$, let $m>j$ and $s$ be large
enough so that $n$ is in the domain and range of $\Psi_s^{G
\rstrd m}$.  Then letting $\tau$ be $G \rstrd m$ we have $G
\rstrd m \in R_s$, contradicting $\Psi^G$ being an isomorphism.
We conclude that $G$ does not meet $S$.

By genericity, there is an $l$ such that for all $\sigma \in T$
with $\sigma \supseteq G \rstrd l$ we have $\sigma \notin S$.
Hence we have:
\begin{eqnarray} & \forall \sigma \in T \, [\sigma \supseteq G \rstrd l \to \forall n \, \exists s \, \exists \tau \in T & \nonumber\\ & [\tau \supseteq \sigma \ \wedge\ \tau \notin R_s \ \wedge\ n \in \text{dom}(\Psi_s^\tau) \ \wedge\  n \in \text{ran}(\Psi_s^\tau)]] & \label{notS} \end{eqnarray}

We can now construct our isomorphism $f: \sta \to \stb$ with $f
\leq_T T$ to complete the proof.  We will $T$-computably build
$A = \bigcup_{i \in \omega} \alpha_i$ so that $f = \Psi^A$.

Let $\alpha_0$ be $G \rstrd l$.  Given $\alpha_i$, let
$\alpha_{i+1}$ be the first $\tau$ we find satisfying
(\ref{notS}) with $i$ for $n$ and $\alpha_i$ for $\sigma$.  We
note that every $\alpha_i \supseteq G \rstrd l$ (and $\alpha_i
\in T$), so finding a $\tau$ which satisfies (\ref{notS}) is
always possible.  This completes our construction.

We note $A \leq_T T$ so $f \leq_T T$.  From the construction it
is clear that $f$ is total and surjective.  To show that $f$ is
an isomorphism, it suffices to show that for all $i$ and $t$ we
have $\alpha_i \notin R_t$.

Suppose $\alpha_i \in R_t$ for some $i,t$.  Let $j>i$ be
sufficiently large such that $j \in$ dom($\Psi_s^A$) requires
$s>t$.  We then have $\alpha_j \notin R_s$ for some $s>t$ so
$\alpha_j \notin R_t$.  Since $\alpha_i \subseteq \alpha_j$ we
have $\alpha_i \notin R_t$ for a contradiction.  We conclude
that for all $i,t$, we have $\alpha_i \notin R_t$.  Thus $f$ is
an isomorphism from $\sta$ to $\stb$.

As noted at the start of the proof, this implies $G \leq_T T$
for a contradiction.  We conclude $G$ is not a degree of
categoricity.
\end{proof}

\begin{cor} Let $A$ be a set and let $G$ be 2-generic($A$).
Then the degree of $G \oplus A$ is not a degree of categoricity. \label{GplusA} \end{cor}
\begin{proof} Let $T = \{ \sigma \in \strings \setsep \exists
\tau \in \strings \, [\sigma \subseteq \tau \oplus A]\}$. Then
$G \oplus A$ is 2-generic relative to $T$ so by Theorem
\ref{Grpt}, the degree of $G \oplus A$ is not a degree of
categoricity.
\end{proof}

\begin{cor} Let $\dg{x}$ be any Turing degree. Then there exists
$\dg{y}$ with\\$\dg{x} \leq_T \dg{y} \leq_T \dg{x^{\prime\prime}}$
such that $\dg{y}$ is not a degree of categoricity. \end{cor}
\begin{proof} Let $X$ be a set of degree $\dg{x}$ and let
$G \leq_T X^{\prime \prime}$ be 2-generic($X$). Let $\dg{y}$
be the degree of $X \oplus G$.  Then $\dg{x} \leq_T \dg{y}
\leq_T \dg{x^{\prime\prime}}$ and by Corollary \ref{GplusA},
$\dg{y}$ is not a degree of categoricity.
\end{proof}

\section{Hyperimmune-free}

Recall that a degree $\dg{y}$ is hyperimmune-free if every
function $f \leq_T \dg{y}$ can be bounded by a computable
function (see \cite{Soare}).  We note that all known degrees of
categoricity $\dg{x}$ are such that $\dg{0}^{(\gamma)} \leq_T
\dg{x} \leq_T \dg{0}^{(\gamma + 1)}$ for some ordinal $\gamma$,
and hence are hyperimmune (or computable).  This suggests the
question, is there a (noncomputable) degree of categoricity
which is hyperimmune-free?  We show that no such degree exists.

\begin{theorem} Let $\dg{b}$ be a noncomputable hyperimmune-free degree.  Then $\dg{b}$ is not a degree of categoricity. \end{theorem}
\begin{proof} Let $\dg{b}$ be a noncomputable hyperimmune-free degree, and assume
for a contradiction that $\dg{b}$ is a degree of categoricity.
Let $\sta$ witness that $\dg{b}$ is a degree of categoricity.
Let $\stb$ be an arbitrary computable structure such that
$\sta$ is isomorphic to $\stb$. We will show there is an
isomorphism $g: \sta \to \stb$ such that $g \leq_T \es^\prime$.
Since $\stb$ is arbitrary, we will then have $\dg{0'} \in$
CatSpec($\sta$). Hence $\dg{b} \leq \dg{0'}$, contradicting
$\dg{b}$ being noncomputable and hyperimmune-free.  Therefore
it suffices to show there exists such a $g$.

Let $f:\omega \to \omega$ be an isomorphism from $\sta$ to
$\stb$ with $f \leq_T \dg{b}$.  We note since $f$ is bijective,
$f^{-1} \leq_T \dg{b}$.  Since $\dg{b}$ is hyperimmune-free,
let $h$ be a computable function which dominates $f$ and
$f^{-1}$.

We now use $h$ to build an infinite computably bounded tree $T
\subset \omega^{< \omega}$ whose infinite paths code
isomorphisms between $\scrA$ and $\scrB$. Then $[T]$ must have
a $\es'$-computable member (indeed, a low member), so there
exists $g \leq_T \es'$ with $g : \scrA \iso \scrB$ as desired.

The infinite paths through $T$ will code isomorphisms by having
the map from $\scrA$ to $\scrB$ on the even bits, and the
inverse map on the odd bits. For $\sigma \in \baistr$, let
$\sigma_0(n) = \sigma(2n)$ and $\sigma_1(n) = \sigma(2n +1)$.
Let $T$ be defined by:
\begin{eqnarray*}T = \{ \sigma \in \baistr \setsep \forall n \leq \len
(\sigma) [[\sigma(n) \leq h(\lfloor\frac{n}{2}\rfloor)] \wedge
\sigma_0 \in \mbox{Part}(\scrA, \scrB) \wedge \\
\sigma_1 \in \mbox{Part}(\scrB, \scrA) \wedge (i\neq j \rightarrow
(\sigma_i(\sigma_j(n)) = n \vee \sigma_i(\sigma_j(n))\up))] \}\end{eqnarray*}
Then $T$ is a computably bounded tree. Let $\tilde{f}(2n) =
f(n)$ and $\tilde{f}(2n+1) = f^{-1} (n)$. Then $\tilde{f} \in
[T]$, so $T$ is infinite. Let $\tilde{g} \in T$ be such that
$\tilde{g} \leq_T \es'$. Let $g(n) = \tilde{g}(2n)$. Then $g
\leq_T \tilde{g} \leq_T \es'$ and $g: \scrA \iso \scrB$ as
desired.
\end{proof}

\section{$\Sigma^0_2$ degree} \label{sigsec}

\begin{theorem} There is a $\Sigma^0_2$ degree that is not a degree of categoricity. \label{sigthm} \end{theorem}
\begin{proof}
We build a set $D$ with $\Sigma^0_2$-degree $\dg{d}$ that is
not a degree of categoricity. This time, instead of a $\es''$-oracle construction, we build
$D$ to be left-\ce in $\es'$.  We again meet the requirements:

$R_{\la e, i, j \ra}:$ If $\Phi_e^D: \scrA_i \iso \scrA_j$ then
there is a computable isomorphism between $\scrA_i$ and
$\scrA_j$.

If, for example, $R_{\la e, i, j\ra}$ were the highest priority
requirement, we would ask $\es'$, $(\exists \sigma \supseteq 1)
(\exists t)( \Phi_{e,t}^\sigma\nin$
Part$(\scrA_{i,t},\scrA_{j,t})$? If yes, we would extend to
$\sigma$.

If no, then (while letting $\delta_0 = 0$ and also addressing
other requirements) at all subsequent stages $s$ we would ask
$\es'$, $(\forall \sigma \supseteq 1, |\sigma| = s)( \forall k
\leq s)( \exists \tau \supseteq \sigma)( \exists t) (k \in
\text{dom}\Phi^\tau_{e,t}\wedge  k \in
\text{rng}\Phi^\tau_{e,t})?$

If the answer is always ``yes'', then we can use $\Phi_e$ to build
a computable isomorphism between $\scrA_i$ and $\scrA_j$.

If the answer is ``no" at some stage, then at that stage we set
$\delta_s = \sigma$ such that $\Phi_e^\sigma$ cannot be
extended to an isomorphism. This is a move that is left-c.e. in
$\es'$, and causes injury to lower priority requirements.

We now give the formal construction. Of course, in addition to
the $R_{\la e, i, j\ra}$ requirements discussed above, we must
also meet non-computability requirements:

$N_e: D \not=\varphi_e$.

At each stage $s$, requirements of the form $R_{\la e, i, j
\ra}$ will either be \emph{unsatisfied}, \emph{under
consideration}, or \emph{satisfied}. Requirements of the form
$N_e$ will either be \emph{unsatisfied} or \emph{satisfied}.
The status of each requirement will change at most finitely
often. We will have $\delta_s$ denote the stage $s$
approximation to $D$. We will either have $\delta_{s+1} \supset
\delta_s$, or $\delta_{s+1} \supseteq \sigma_{\la e, i, j\ra }$
where $\delta_s \supseteq \delta_k\conc 0$ and $\sigma_{\la e,
i, j\ra} \supseteq \delta_k\conc 1$ for some $k <s$, so that
the approximation will be left-\ce in $\es'$.

\emph{Stage 0}: Set $\delta_0 = \es$.

\emph{Stage s+1=2m +1}: For each $R_{\la e,i,j\ra}$ currently
under consideration, in turn, ask $\es'$, $(\forall \sigma
\supseteq \sigma_{\la e, i, j\ra}, |\sigma| = s)( \forall k
\leq s)( \exists \tau \supseteq \sigma)( \exists t) (k \in
\text{dom}\Phi^\tau_{e,t}\wedge  k \in
\text{rng}\Phi^\tau_{e,t})?$ If there is a least $R_{\la e,i, j
\ra }$ for which the answer is ``no", then let $\delta_{s+1}
\supseteq \sigma_{\la e, i, j\ra}$ be such that $(\exists
k)(\forall \tau \supseteq \delta_{s+1})(k \nin
\text{dom}\Phi^\tau_{e}\vee k \nin \text{rng}\Phi^\tau_{e})$.
Declare $R_{\la e, i, j\ra}$ to be satisfied. Declare $N_n$ and
$R_n$ to be unsatisfied for all $n > \la e,i, j\ra$, canceling
any associated $\sigma_n$ for those $R_n$ that were under
consideration.

If the answer was always ``yes", let $\la e, i, j \ra$ be least
such that $R_{\la e, i, j\ra}$ is unsatisfied. Ask $\es'$,
$(\exists \sigma \supseteq \delta_s\conc1) (\exists t)(
\Phi_{e,t}^\sigma\!\!\nin\,$Part$(\scrA_{i,t},\scrA_{j,t})$? If the
answer is ``yes", let $\delta_{s+1}$ have this property, and
declare $R_{\la e,i,j \ra}$ to be satisfied. If the answer is
``no", let $\sigma_{\la e, i, j \ra} = \delta_s\conc1$, let
$\delta_{s+1}=\delta_s\conc0$, and declare $R_{\la e, i, j\ra}$
to be under consideration.

\emph{Stage s+1=2m +2}: Let $n$ be least such that $N_n$ is not
satisfied. Use $\es'$ to determine whether $(\exists
t)(\varphi_{n,t}(|\delta_s|)= 0)$. If yes, define $\delta_{s+1}
= \delta_s\conc1$. Otherwise, define
$\delta_{s+1}=\delta_s\conc0$. Declare $N_n$ to be satisfied.

This completes the construction.

\begin{lemma}The approximation $\{\delta_s\}_{s\in \omega}$
defines a set $D$ that is left-\ce in $\es'$.
\end{lemma}
\begin{proof}
The construction is $\es'$-computable, so the sequence
$\{\delta_s\}_{s\in \omega}$ is $\es'$-computable. We either
have $\delta_{s+1} \supset \delta_s$, or we have $\delta_{s+1}
\supseteq \sigma_{\la e, i, k\ra}$ for some $R_{\la e, i,
k\ra}$ that was under consideration at stage $s+1$. In the
first case, certainly $\delta_{s} <_L \delta_{s+1}$. In the
second case, note that at the greatest stage $t+1\leq s$ where
$\sigma_{\la e, i, j \ra}$ was defined, we had let
$\delta_{t+1} = \delta_t\conc0$ and $\sigma_{\la e, i, j \ra}=
\delta_t \conc1$. As $R_{\la e, i, j\ra}$ remained under
consideration between stages $t+1$ and $s+1$, there was no
shift to $\sigma_k$ for any $k \leq \la e, i, j \ra$ between
stages $t+1$ and $s$. It is easy to see by induction on $t+1
\leq t' \leq s$ that $\delta_{t'} \supseteq \delta_{t+1}$ and
that if $\sigma_n$ was defined at stage $t'$ then $\sigma_n
\supseteq \delta_{t'}$. So $\delta_{s} \supseteq \delta_t
\conc0$ and $\delta_{s+1} \supseteq \delta_t\conc1$. That is,
the approximation $\{\delta_s\}_{s\in \omega}$ is left-\ce in
$\es'$.
\end{proof}

\begin{lemma} For each $n$, there is a stage $s$ after which
the status of requirement $R_n$ ceases to change. Each
requirement $R_n$ is met.
\end{lemma}
\begin{proof}
Consider the requirement $R_{\la e, i, j \ra}$, and let $s$ be
the least stage by which the status of all $R_n$ for $n< \la e,
i, j \ra$ cease to change. Note that $R_{\la e, i, j \ra}$ had
status ``unsatisfied" at stage $s$. At stage $s+1$, the status
of $R_{\la e, i, j \ra}$ becomes ``satisfied" or ``under
consideration". The only way for the status of $R_{\la e, i,
j\ra}$ to change back to unsatisfied would be if there was a
change in status for some $R_n$ with $n < \la e, i, j \ra$. So
$R_{\la e, i, j\ra}$ is never again unsatisfied.

Suppose there is a stage greater than $s$ where $R_{\la e, i, j
\ra}$ becomes satisfied. Let $t$ be the least such stage. At
stage $t$ we set $\delta_t$ such that $\Phi_e^{\delta_t} \nin
\text{Part}(\scrA_{i,t}, \scrA_{j,t})$, or such that there
exists $k$ such that for all $\tau \supseteq \gamma_t$, $k \nin
\text{dom}\Phi^\tau_e$ or $k \nin \text{rng}\Phi^\tau_e$. Since
there is no change in status for any $R_n$ with $n < \la e, i,
j \ra$ beyond stage $s$, it is easy to see by induction on $t'
> t$ that $\delta_{t'} \supseteq \delta_{t+1}$ and that if
$\sigma_n$ was defined at stage $t'$ then $\sigma_n \supseteq
\delta_{t'}$. So $D \supset \delta_t$ and $\Phi_e^D \nin
\text{Part}(\scrA_i, \scrA_j)$, i.e., $R_{\la e, i, j \ra}$ is
met, and indeed has the status satisfied at all stages beyond
$t$.

Suppose there is no stage greater than $s$ where $R_{\la e, i,
j\ra}$ becomes satisfied. In this case, $R_{\la e, i, j\ra}$
maintains status under consideration at all stages greater than
$s$, and $\sigma_{\la e, i, j \ra}$ does not change after it is
defined at stage $s+1$. We now show that $R_{\la e, i, j \ra}$
is met, by building a computable isomorphism $f: \scrA_i \rt
\scrA_j$. First note that $(\forall \sigma \supset \sigma_{\la
e, i, j \ra})(\forall t)(\Phi_{e,t}^\sigma \in
\text{Part}(\scrA_{i,t}, \scrA_{j,t}))$. Since $R_{\la e, i, j
\ra}$ remains under consideration at all stages $n > s$,
$(\forall \sigma \supseteq \sigma_{\la e, i, j\ra}, |\sigma| =
n)( \forall k \leq n)( \exists \tau \supseteq \sigma)( \exists
t) (k \in \text{dom}\Phi^\tau_{e,t}\wedge  k \in
\text{rng}\Phi^\tau_{e,t})$. Let $\tau_0 = \sigma_{\la e, i,
j\ra}$. Given $\tau_n$, let $\tau_{n+1} \supset \tau_n$ be such
that $n \in \text{dom}\Phi^\tau_{e}$ and $n \in
\text{rng}\Phi^\tau_{e}$, and $|\tau_{n+1}| \geq n+1$. Since we
know such $\tau_{n+1}$ exists, we can search and find such
effectively. Let $B= \cup_{n \in \omega} \tau_n$. Then $B$ is
computable, and $\Phi_e^B: \scrA_i \iso \scrA_j$.

\end{proof}
\begin{lemma} The requirements $N_n$ are all met.
\end{lemma}
\begin{proof}
We prove by induction on $n$ that if $s$ is the least stage
when all requirements $R_i$ and $N_j$ for $i \leq n$, $j < n$ cease
to change status, then $N_n$ is satisfied at all stages beyond
stage $s+1$. Indeed, suppose the
result holds for $k <n$, and that $R_n$ obtains its final
status for the first time at stage $s$. At stage $s+1$, $N_n$
is either already satisfied, or receives attention and becomes
satisfied. Since all the $R_k$ for $k \leq n$ do not change
status after stage $s$, it is easy to see that $D \supset
\delta_{s+1}$, so $N_n$ remains satisfied beyond stage $s+1$.
\end{proof}

\end{proof}

\begin{remark} The set $D$ constructed in Theorem \ref{sigthm} is such that $D \not\geq_T \es^\prime$. \end{remark}
\begin{proof} Fokina, Kalimullin, and Miller \cite{FKM} showed
that all degrees which are c.e.a.\ $(\es^\prime)$ are degrees
of categoricity. If we had $D \geq_T \es^\prime$ then $D$ would
be c.e.a.\ $(\es^\prime)$ for a contradiction.
\end{proof}

We note that since all $\Sigma^0_2$ degrees are hyperimmune, this implies there is a hyperimmune degree which is not a degree of categoricity \cite{refnote}.

\section{Conclusion}

Considerable ground remains open in finding how low in
complexity a degree can be without being a degree of
categoricity.  On one side, it is not known if there is a
$\Delta^0_2$ degree which is not a degree of categoricity.  On
the other side, it has not been shown that every 3-c.e.\ degree
is a degree of categoricity.

We can also consider questions about categorizing the degrees
of categoricity.  What other classes of degrees can be shown to
lack (or have) degrees of categoricity?  Must every degree of
categoricity $\dg{x}$ be such that $\dg{0}^{(\gamma)} \leq_T
\dg{x} \leq_T \dg{0}^{(\gamma + 1)}$ for some ordinal $\gamma$?

Finally, the connection between degrees of categoricity and
strong degrees of categoricity can be further explored.  The
question of Fokina, Kalimullin, and Miller \cite{FKM}, is there
a degree of categoricity which is not strong, remains open.  In
fact, every computable structure constructed so far that
witnesses a degree of categoricity, does so by witnessing a
strong degree of categoricity.

\nocite{*}
\bibliography{NotDegreesofCatFinal}
\end{document}